\documentclass[11pt,reqno]{amsart}

\usepackage[left=3.5cm,right=3.5cm]{geometry}
\usepackage{amssymb}
\usepackage{graphicx}
\usepackage{amscd}
\usepackage[hidelinks]{hyperref}
\usepackage{color}
\usepackage{float}
\usepackage{graphics,amsmath,amssymb}
\usepackage{amsthm}
\usepackage{amsfonts}
\usepackage{latexsym}
\usepackage{epsf}
\usepackage{xifthen}
\usepackage{mathrsfs}
\usepackage{dsfont}
\usepackage{makecell}
\usepackage[FIGTOPCAP]{subfigure}
\usepackage{amsmath}
\allowdisplaybreaks[4]
\usepackage{listings}
\usepackage{etoolbox}
\usepackage{fancyhdr}
\usepackage{pdflscape}
\usepackage[title,toc,titletoc]{appendix}
\usepackage{enumitem}
\usepackage[noadjust]{cite}
\usepackage{tikz}
\usepackage{young}

\numberwithin{equation}{section}

\newtheorem{theorem}{Theorem}[section]
\newtheorem*{theorem*}{Theorem}

\newtheorem{corollary}[theorem]{Corollary}
\newtheorem{lemma}[theorem]{Lemma}
\newtheorem{proposition}[theorem]{Proposition}

\newtheorem{problem}[theorem]{Problem}

\newtheorem{innercustomgeneric}{\customgenericname}
\providecommand{\customgenericname}{}
\newcommand{\newcustomtheorem}[2]{%
    \newenvironment{#1}[1]
    {%
        \renewcommand\customgenericname{#2}%
        \renewcommand\theinnercustomgeneric{##1}%
        \innercustomgeneric
    }
    {\endinnercustomgeneric}
}
\newcustomtheorem{ctheorem}{Theorem}
\newcustomtheorem{clemma}{Lemma}

\theoremstyle{definition}
\newtheorem{definition}[theorem]{Definition}

\newtheorem{example}[theorem]{Example}
\newtheorem*{example*}{Example}
\newtheorem*{examples*}{Examples}
\newtheorem{remark}[theorem]{Remark}
\newtheorem*{remark*}{Remark}
\newtheorem*{remarks*}{Remarks}

\newtheoremstyle{named}{}{}{\itshape}{}{\bfseries}{.}{.5em}{#1\thmnote{ #3}}
\theoremstyle{named}

\newcommand{\mex}{\operatorname{mex}}

\title[Moments of the minimal partition excludant]{Asymptotics for moments of the minimal partition excludant
in congruence classes}

\author[S. Chern]{Shane Chern}
\address[S. Chern]{Department of Mathematics
 and Statistics, Dalhousie University, Halifax, NS, B3H 4R2, Canada}
\email{chenxiaohang92@gmail.com}

\author[E. X. W. Xia]{Ernest X. W. Xia}
\address[E. X. W. Xia]{School of Mathematical Sciences, Suzhou University of Science and Technology, Suzhou, 215009, Jiangsu Province, P. R. China}
\email{ernestxwxia@163.com}

\date{}



\begin{document}

\maketitle

\begin{abstract}
    The minimal excludant statistic, which denotes the smallest positive integer that is not a part of an integer partition, has received great interest in recent years. In this paper, we move on to the smallest positive integer whose frequency is less than a given number. We establish an asymptotic formula for the moments of such generalized minimal excludants that fall in a specific congruence class. In particular, our estimation reveals that the moments associated with a fixed modulus are asymptotically ``equal''.
\end{abstract}

\section{Introduction}

For a given natural number $n$, we call a collection of weakly
decreasing positive integers $\pi = (\pi_1,\pi_2,\ldots,\pi_\ell)$ a
\emph{partition} of $n$ if $|\pi|:=\sum \pi_i = n$, while the
numbers $\pi_i$ are called \emph{parts} of this partition. We use
the notation $\pi \vdash n$ to represent that $\pi$ is a partition
of $n$.

The study of integer partitions and their counting traces back to
Leibniz in his correspondence with Jacob Bernoulli
\cite[pp.~740--755]{Lei1990}, and stays vibrant in modern number
theory and combinatorics, partially through various statistics
associated with partitions. One recent example is the \emph{minimal
excludant} statistic, abbreviated \emph{mex}, which was named by
Andrews and Newman \cite{AD2019} based on a terminology in
combinatorial game theory \cite{FP2015}.

\begin{definition}
    The \emph{minimal excludant} of an integer partition $\pi$, denoted by $\mex(\pi)$, is the smallest positive integer such that it is not a part of $\pi$.
\end{definition}

Notably, the study of this statistic was launched earlier by Andrews
himself in \cite{And2011} under the name of ``the smallest part that
is not a summand,'' which is semantically more accessible, and even
earlier by Grabner and Knopfmacher \cite{GK2006}, who called it
``the least gap.'' In recent years, the mex statistic has been found
to exhibit connections with a handful of objects such as Dyson's
crank \cite{HSS2022,HSY2022}, Schmidt-type partitions
\cite{WX2023,WZZ2023}, unrefinable partitions
\cite{ACCL2022,ACGS2022}, and the super Yang--Mills theory
\cite{HO2024}.

Recall that the \emph{frequency} of a part $\pi_i$ is the number of
its occurrence in the partition $\pi$. Hence, the mex of $\pi$ can
be paraphrased as the smallest positive integer whose frequency is
less than one. Along this line, we may generalize the mex statistic
as follows.

\begin{definition}
    Letting $s$ be a positive integer, we define the \emph{minimal excludant of frequency $s$} (abbr.~\emph{$s$-mex}) of an integer partition $\pi$, denoted by $\mex^{[s]}(\pi)$, as the smallest positive integer such that its frequency in $\pi$ is less than $s$.
\end{definition}

With this definition, the $1$-mex reduces to the usual mex
statistic. It is also notable that
 consideration along this line is
not fresh; for example,
 Wagner \cite{Wag2011} considered
probabilistic aspects of this family of statistics, and Ballantine
and Merca \cite{BM2020} entered this topic from a combinatorial
perspective.

\begin{example}
    Considering the partition $\pi=(6,4,3,3,2,2,2,1,1)$, we have
    \begin{align*}
        \mex^{[1]}(\pi) &= 5,\\
        \mex^{[2]}(\pi) &= 4,\\
        \mex^{[s]}(\pi) &= 1 \quad (s\ge 3).
    \end{align*}
\end{example}

Given a combinatorial statistic, it is of general interest to
consider its arithmetic behavior. Therefore, our first objective
revolves around partitions whose $s$-mex falls in a specific
congruence class. In particular, letting $A$ and $M$ with $0< A\le
M$ be integers, we are interested in the following moments:
\begin{align}
    \sigma_{M,A}^{[s]}(r;n) := \sum_{\substack{\pi\vdash n\\ \mex^{[s]}(\pi)\equiv A \bmod M}} \big(\mex^{[s]}(\pi)\big)^r,
\end{align}
where $r$ is a nonnegative integer.

\begin{theorem}\label{th:sigma-asymp}
    For integers $A$ and $M$ with $0< A\le M$, nonnegative integers $r$ and positive integers $s$, we have, as $n\to \infty$,
    \begin{align}\label{eq:sigma-asymp}
        \sigma_{M,A}^{[s]}(r;n) \sim \begin{cases}
            2^{-2}3^{-\frac{1}{2}}M^{-1}n^{-1}e^{\pi \sqrt{\frac{2n}{3}}}, & \text{if $r=0$},\\[6pt]
            2^{\frac{3r-12}{4}} 3^{\frac{r-2}{4}} \pi^{-\frac{r}{2}} M^{-1}s^{-\frac{r}{2}} r\,\Gamma(\tfrac{r}{2})\, n^{\frac{r-4}{4}} e^{\pi \sqrt{\frac{2n}{3}}} , & \text{if $r\ge 1$},
        \end{cases}
    \end{align}
    where $\Gamma(z)$ is Euler's gamma function.
\end{theorem}

As a corollary, we notice that the moments associated with a fixed
modulus are asymptotically ``equal''.

\begin{corollary}\label{eq:sigma-asymp-eq}
    For integers $A$, $A'$ and $M$ with $0< A,A'\le M$, nonnegative integers $r$ and positive integers $s$,
    \begin{align}
        \lim_{n\to \infty} \frac{\sigma_{M,A}^{[s]}(r;n)}{\sigma_{M,A'}^{[s]}(r;n)} = 1.
    \end{align}
\end{corollary}

In \cite{AD2019}, Andrews and Newman also defined the \emph{minimal
odd excludant} to be the smallest odd positive integer that is not a
part of a partition $\pi$. It is notable that this terminology
should not be confused with the minimal excludant that is odd. In a
subsequent paper \cite{AD2020}, Andrews and Newman further extended
this idea to arbitrary congruence classes, by defining
$\mex_{M,A}(\pi)$ as the smallest positive integer congruent to $A$
modulo $M$ such that it is not a part of $\pi$. Now we may similarly
introduce restrictions on the frequency of parts.

\begin{definition}
    Letting $s$ be a positive integer and $A$ and $M$ with $0< A\le M$ be integers, we define the \emph{minimal excludant of frequency $s$ in the congruence class $A$ modulo $M$} (abbr.~\emph{$s_{M,A}$-mex}) of an integer partition $\pi$, denoted by $\mex_{M,A}^{[s]}(\pi)$, as the smallest positive integer congruent to $A$ modulo $M$ such that its frequency in $\pi$ is less than $s$.
\end{definition}

Thus, $\mex_{M,A}^{[1]}(\pi) = \mex_{M,A}(\pi)$. It is also notable
that the $s_{M,A}$-mex reduces to the $s$-mex when the pair $(M,A)$
takes $(1,1)$.

\begin{example}
    Considering the partition $\pi=(6,4,3,3,2,2,2,1,1)$, we have
    \begin{alignat*}{2}
        \mex^{[1]}_{2,1}(\pi) &= 5, &\qquad \mex^{[1]}_{2,2}(\pi) &= 8,\\
        \mex^{[2]}_{2,1}(\pi) &= 5, & \mex^{[2]}_{2,2}(\pi) &= 4,\\
        \mex^{[3]}_{2,1}(\pi) &= 1, & \mex^{[3]}_{2,2}(\pi) &= 4,\\
        \mex^{[s]}_{2,1}(\pi) &= 1, & \mex^{[s]}_{2,2}(\pi) &= 2 \quad (s\ge 4).
    \end{alignat*}
\end{example}

Our second objective focuses on the moments of the $s_{M,A}$-mex:
\begin{align}
    \varsigma_{M,A}^{[s]}(r;n) := \sum_{\pi\vdash n} \big(\mex_{M,A}^{[s]}(\pi)\big)^r,
\end{align}
where $r$ is a nonnegative integer. It is clear that when $r=0$,
\begin{align}\label{eq:varsigma-0-p}
    \varsigma_{M,A}^{[s]}(0;n) = \sum_{\pi\vdash n} 1 = p(n),
\end{align}
where the \emph{partition function} $p(n)$ counts the number of
partitions of $n$. It is a standard result due to Hardy and
Ramanujan \cite[p.~79, eq.~(1.41)]{HR1918} that
\begin{align*}
    p(n)\sim 2^{-2}3^{-\frac{1}{2}}n^{-1}e^{\pi \sqrt{\frac{2n}{3}}}.
\end{align*}
Hence, for any choice of $A$, $M$ and $s$,
\begin{align}
    \varsigma_{M,A}^{[s]}(0;n) \sim 2^{-2}3^{-\frac{1}{2}}n^{-1}e^{\pi \sqrt{\frac{2n}{3}}}.
\end{align}
In what follows, we focus on the cases where $r\ge 1$.

\begin{theorem}\label{th:varsigma-asymp}
    For integers $A$ and $M$ with $0< A\le M$, \textbf{positive} integers $r$ and positive integers $s$, we have, as $n\to \infty$,
    \begin{align}\label{eq:varsigma-asymp}
        \varsigma_{M,A}^{[s]}(r;n) \sim 2^{\frac{3r-12}{4}} 3^{\frac{r-2}{4}} \pi^{-\frac{r}{2}} M^{\frac{r}{2}} s^{-\frac{r}{2}} r\,\Gamma(\tfrac{r}{2})\, n^{\frac{r-4}{4}} e^{\pi \sqrt{\frac{2n}{3}}}.
    \end{align}
\end{theorem}

We have similar ``equalities'' of these moments in the asymptotic
sense.

\begin{corollary}\label{eq:varsigma-asymp-eq}
    For integers $A$, $A'$ and $M$ with $0< A,A'\le M$, positive integers $r$ and positive integers $s$,
    \begin{align}
        \lim_{n\to \infty} \frac{\varsigma_{M,A}^{[s]}(r;n)}{\varsigma_{M,A'}^{[s]}(r;n)} = 1.
    \end{align}
\end{corollary}

\section{Generating functions}

It is known, according to Euler \cite[p.~260, Ch.~XVI]{Eul1988},
that the generating function of the partition function $p(n)$ is
\begin{align*}
    \sum_{n\ge 0} p(n) q^n = \prod_{k\ge 1}\frac{1}{1-q^k} = \frac{1}{(q;q)_\infty},
\end{align*}
wherein we adopt the \emph{$q$-Pochhammer symbol} for complex $q$
with $|q|<1$:
\begin{align*}
    (a;q)_\infty := \prod_{k\ge 0} (1-aq^k).
\end{align*}

Now we look into the generating functions of
$\sigma_{M,A}^{[s]}(r;n)$ and $\varsigma_{M,A}^{[s]}(r;n)$ \`a la
Andrews and Newman \cite{AD2019,AD2020}. Throughout this section, we
assume that $A$ and $M$ with $0< A\le M$ are integers, $r$ is a
nonnegative integer and $s$ is a positive integer.

We first work on $\sigma_{M,A}^{[s]}(r;n)$. It is clear that
\begin{align*}
    &\sum_{n\ge 0}\sigma_{M,A}^{[s]}(r;n) q^n\\
    &\qquad = \sum_{n\ge 0} (Mn+A)^r\big(1+q^{Mn+A}+\cdots + q^{(s-1)(Mn+A)}\big)\\
    &\qquad\quad\times \left(\prod_{k>Mn+A}\frac{1}{1-q^k}\right) \left(\prod_{1\le k<Mn+A}\frac{q^{sk}}{1-q^k}\right)\\
    &\qquad=\sum_{n\ge 0} (Mn+A)^r\frac{1-q^{s(Mn+A)}}{1-q^{Mn+A}}\cdot q^{s(1+2+\cdots (Mn+A-1))} \prod_{\substack{k\ge 1\\k\ne Mn+A}}\frac{1}{1-q^k}.
\end{align*}
Thus the generating function of $\sigma_{M,A}^{[s]}(r;n)$ can be
formulated as follows.

\begin{proposition}
    We have
    \begin{align}\label{eq:sigma-gf}
        \sum_{n\ge 0} \sigma_{M,A}^{[s]}(r;n)q^n = \frac{1}{(q;q)_\infty} \sum_{n\ge 0} (Mn+A)^r q^{\frac{s}{2}(Mn+A)(Mn+A-1)} (1-q^{s(Mn+A)}).
    \end{align}
\end{proposition}

For $\varsigma_{M,A}^{[s]}(r;n)$, we have
\begin{align*}
    &\sum_{n\ge 0}\varsigma_{M,A}^{[s]}(r;n) q^n\\
    &\qquad = \prod_{\substack{k\ge 1\\k\not\equiv A \bmod{M}}}\frac{1}{1-q^k}\cdot \sum_{n\ge 0} (Mn+A)^r\big(1+q^{Mn+A}+\cdots + q^{(s-1)(Mn+A)}\big)\\
    &\qquad\quad\times \left(\prod_{k>n}\frac{1}{1-q^{Mk+A}}\right)\left(\prod_{0\le k<n}\frac{q^{s(Mk+A)}}{1-q^{Mk+A}}\right)\\
    &\qquad=\prod_{\substack{k\ge 1\\k\not\equiv A \bmod{M}}}\frac{1}{1-q^k}\cdot \sum_{n\ge 0} (Mn+A)^r\frac{1-q^{s(Mn+A)}}{1-q^{Mn+A}}\\
    &\qquad\quad\times q^{s(A+(M+A)+\cdots+(M(n-1)+A))} \prod_{\substack{k\ge 0\\k\ne n}}\frac{1}{1-q^{Mk+A}}.
\end{align*}

\begin{proposition}
    We have
    \begin{align}\label{eq:varsigma-gf}
        \sum_{n\ge 0} \varsigma_{M,A}^{[s]}(r;n)q^n = \frac{1}{(q;q)_\infty} \sum_{n\ge 0} (Mn+A)^r q^{s(\frac{1}{2}Mn(n-1)+An)} (1-q^{s(Mn+A)}).
    \end{align}
\end{proposition}

\begin{remark}
    It is clear that when $r=0$,
    \begin{align*}
        &\sum_{n\ge 0} q^{s(\frac{1}{2}Mn(n-1)+An)} (1-q^{s(Mn+A)})\\
        &\qquad = 1 + \sum_{n\ge 1} q^{s(\frac{1}{2}Mn(n-1)+An)} - \sum_{n\ge 0} q^{s(\frac{1}{2}Mn(n+1)+A(n+1))} = 1.
    \end{align*}
    Hence, the relation $\varsigma_{M,A}^{[s]}(0;n)=p(n)$ as given in \eqref{eq:varsigma-0-p} is recovered. In general, the sum on the right-hand side of \eqref{eq:varsigma-gf} can be rewritten as
    \begin{align}\label{eq:varsigma-gf-sum}
        &\sum_{n\ge 0} (Mn+A)^r q^{s(\frac{1}{2}Mn(n-1)+An)} (1-q^{s(Mn+A)})\notag\\
        &\qquad = A^r + \sum_{n\ge 0} (M(n+1)+A)^r q^{s(\frac{1}{2}Mn(n+1)+A(n+1))}\notag\\
        &\qquad\quad - \sum_{n\ge 0} (Mn+A)^r q^{s(\frac{1}{2}Mn(n+1)+A(n+1))}.
    \end{align}
    This reformulation will facilitate our proof of Theorem \ref{th:varsigma-asymp}.
\end{remark}

\section{Lemmas}

Here we collect some lemmas for our asymptotic analysis in the next
section. Throughout, the \emph{Bernoulli polynomials} $B_n(x)$ are
defined by the exponential generating function:
\begin{align*}
    \sum_{n\ge 0} B_n(x)\frac{t^n}{n!} = \frac{te^{xt}}{e^t - 1}.
\end{align*}

We begin with an asymptotic estimation for weighted partial theta
functions.

\begin{lemma}\label{le:asymp-partial}
    Let $u$ be a positive real number and let $r$ be a nonnegative integer. As $t\to 0^+$, it is true for any $N\ge 1$ that
    \begin{align}\label{eq:asymp-partial}
        \sum_{n\ge 0} (n+u)^r e^{-(n+u)^2t^2} = \frac{\Gamma(\tfrac{r+1}{2})}{2t^{r+1}}-\sum_{n= 0}^{N-1}\frac{(-1)^n B_{2n+r+1}(u)}{(2n+r+1)n!}t^{2n} + O(t^{2N}).
    \end{align}
\end{lemma}

\begin{proof}
    Asymptotic relations of this type were systematically analyzed in the work of Zagier \cite{Zag2006}, who had utilized the Euler--Maclaurin summation formula. In particular, in \cite[p.~321, eq.~(6.76)]{Zag2006}\footnote{This equation should be corrected as
        \begin{align*}
            \sum_{m=0}^\infty f\big((m+a)x\big) \sim \frac{I_f}{x} - \sum_{n=0}^\infty b_n \frac{B_{n+1}(a)}{n+1} x^n,
        \end{align*}
        since \cite[p.~314, eq.~(6.55)]{Zag2006} reads $\zeta(-n,a) = - \frac{B_{n+1}(a)}{n+1}$ so that the sign before the summation on the right-hand side of the above should be negative.
    }, we choose
    \begin{align*}
        f(x) = x^r e^{-x^2} = \sum_{n\ge 0} \frac{(-1)^n}{n!} x^{2n+r},
    \end{align*}
    so that
    \begin{align*}
        \int_0^\infty f(x) dx = \int_0^\infty x^r e^{-x^2} dx = \frac{\Gamma(\tfrac{r+1}{2})}{2}.
    \end{align*}
    Noting the fact that
    \begin{align*}
        \sum_{n\ge 0} (n+u)^r e^{-(n+u)^2t^2} = \frac{1}{t^r} \sum_{n\ge 0} \big((n+u)t\big) e^{-((n+u)t)^2},
    \end{align*}
    we arrive at \eqref{eq:asymp-partial} by invoking Zagier's result.
\end{proof}

Next, we recall Ingham's Tauberian theorem \cite{Ing1941}.

\begin{lemma}\label{le:ing}
    Let $F(q)=\sum_{n\ge0}f(n) q^n$ be a power series with eventually nondecreasing and nonnegative coefficients such that its radius of convergence equals $1$. If there are constants $A>0$ and $\lambda,\alpha\in\mathbb{R}$ such that
    $$F(e^{-t}) \sim\lambda\, t^\alpha e^{\frac{A}{t}}$$
    as $t\to 0^+$, then
    $$f(n)\sim \frac{\lambda}{2\sqrt{\pi}}\frac{A^{\frac{\alpha}{2}+\frac14}}{n^{\frac{\alpha}{2}+\frac34}}\, e^{2\sqrt{An}}$$
    as $n\to \infty$.
\end{lemma}

\section{Asymptotic analysis}

In this section, we first prove Theorem \ref{th:sigma-asymp}.

\begin{proof}[Proof of Theorem \ref{th:sigma-asymp}]
    We first examine that the moments $\sigma_{M,A}^{[s]}(r;n)$ are eventually nondecreasing and nonnegative. The nonnegativity is apparent as
    \begin{align*}
        \sum_{n\ge 0} \sigma_{M,A}^{[s]}(r;n)q^n = \sum_{n\ge 0} (Mn+A)^r q^{\frac{s}{2}(Mn+A)(Mn+A-1)} \prod_{\substack{k\ge 1\\k\ne s(Mn+A)}} \frac{1}{1-q^k},
    \end{align*}
    wherein each summand is a series with nonnegative coefficients. To see the eventual monotonicity, we notice that when $s\ge 2$, the number $s(Mn+A)$ cannot be $1$ for any nonnegative $n$, thereby implying that
    \begin{align*}
        &(1-q)\cdot \sum_{n\ge 0} \sigma_{M,A}^{[s]}(r;n)q^n\\
        &\qquad = \sum_{n\ge 0} (Mn+A)^r q^{\frac{s}{2}(Mn+A)(Mn+A-1)} \prod_{\substack{k\ge 2\\k\ne s(Mn+A)}} \frac{1}{1-q^k}
    \end{align*}
    is a nonnegative series, and hence that the numbers $\sigma_{M,A}^{[s]}(r;n)$ are nondecreasing in this circumstance. Now we consider the case where $s=1$. Note that
    \begin{align*}
        &\sum_{n\ge 0} \sigma_{M,A}^{[1]}(r;n)q^n \\
        &\quad = \frac{A^rq^{\frac{1}{2}A(A-1)} (1-q^{A})}{(q;q)_\infty} + \frac{1}{(q;q)_\infty} \sum_{n\ge 1} (Mn+A)^r q^{\frac{1}{2}(Mn+A)(Mn+A-1)} (1-q^{Mn+A}).
    \end{align*}
    Clearly, $Mn+A\ne 1$ when $n\ge 1$ because $A$ and $M$ are positive integers. Thus, the latter sum in the above is a series with nondecreasing coefficients since
    \begin{align*}
        &(1-q)\cdot \frac{1}{(q;q)_\infty} \sum_{n\ge 1} (Mn+A)^r q^{\frac{1}{2}(Mn+A)(Mn+A-1)} (1-q^{Mn+A})\\
        &\qquad = \sum_{n\ge 1} (Mn+A)^r q^{\frac{1}{2}(Mn+A)(Mn+A-1)} \prod_{\substack{k\ge 2\\k\ne Mn+A}} \frac{1}{1-q^k}
    \end{align*}
    is a series with nonnegative coefficients. For the remaining term $A^rq^{\frac{1}{2}A(A-1)} (1-q^{A})(q;q)_\infty^{-1}$, we see that when $A\ge 2$,
    \begin{align*}
        (1-q)\cdot \frac{A^rq^{\frac{1}{2}A(A-1)} (1-q^{A})}{(q;q)_\infty} = A^rq^{\frac{1}{2}A(A-1)} \prod_{\substack{k\ge 2\\k\ne A}} \frac{1}{1-q^k}
    \end{align*}
    is a nonnegative series so that $A^rq^{\frac{1}{2}A(A-1)} (1-q^{A})(q;q)_\infty^{-1}$ itself is a series with nondecreasing coefficients. When $A=1$,
    \begin{align*}
        \frac{A^rq^{\frac{1}{2}A(A-1)} (1-q^{A})}{(q;q)_\infty} = \frac{1}{(q^2;q)_\infty} = \sum_{n\ge 0} p_{>1}(n)q^n,
    \end{align*}
    where $p_{>1}(n)$ enumerates the number of \emph{non-unitary partitions} (i.e., partitions with no part equal to one) of $n$. It is notable that $p_{>1}(n+1)\ge p_{>1}(n)$ with only one exception that $p_{>1}(1) = 0 < 1 = p_{>1}(0)$ since for $n\ge 1$ there is a natural injection from non-unitary partitions of $n$ to non-unitary partitions of $n+1$ by adding $1$ to the largest part of the former. Hence, we also have the eventual monotonicity when $A=1$, thereby completing the discussion for all circumstances.

    In view of Ingham's Tauberian theorem, it suffices to estimate the right-hand side of \eqref{eq:sigma-gf} with $q=e^{-t}$ when $t\to 0^+$. To fulfill this goal, we start with a reformulation of the sum on the right-hand side of \eqref{eq:sigma-gf} as
    \begin{align*}
        \sum_{n\ge 0} (Mn+A)^r q^{\frac{s}{2}(Mn+A)(Mn+A-1)} (1-q^{s(Mn+A)}) = S_1(q) - S_2(q),
    \end{align*}
    where
    \begin{align*}
        S_1(q) &:= \sum_{n\ge 0} (Mn+A)^r q^{\frac{s}{2}(Mn+A)(Mn+A-1)},\\
        S_2(q) &:= \sum_{n\ge 0} (Mn+A)^r q^{\frac{s}{2}(Mn+A)(Mn+A+1)}.
    \end{align*}

    When $r=0$, we see that
    \begin{align*}
        S_1(e^{-t}) &= e^{-\frac{s}{2}A(A-1)t} + e^{\frac{st}{8}} \sum_{n\ge 1} e^{-\frac{s}{2}(Mn+A-\frac{1}{2})^2 t}\\
        &= e^{-\frac{s}{2}A(A-1)t} + e^{\frac{st}{8}} \sum_{n\ge 0} e^{-\frac{sM^2}{2}(n+1+\frac{A}{M}-\frac{1}{2M})^2 t}\\
        &= \tfrac{1}{M}\sqrt{\tfrac{\pi}{2}}\,s^{-\frac{1}{2}}t^{-\frac{1}{2}} + 1-B_1\big(1+\tfrac{A}{M}-\tfrac{1}{2M}\big) + O(t^{\frac{1}{2}}),
    \end{align*}
    and that
    \begin{align*}
        S_2(e^{-t}) &= e^{\frac{st}{8}} \sum_{n\ge 0} e^{-\frac{sM^2}{2}(n+\frac{A}{M}+\frac{1}{2M})^2 t}\\
        &= \tfrac{1}{M}\sqrt{\tfrac{\pi}{2}}\,s^{-\frac{1}{2}}t^{-\frac{1}{2}} -B_1\big(\tfrac{A}{M}+\tfrac{1}{2M}\big) + O(t^{\frac{1}{2}}).
    \end{align*}
    We have applied Lemma \ref{le:asymp-partial} for both asymptotic relations above as $t\to 0^+$. Hence,
    \begin{align*}
        \sum_{n\ge 0} q^{\frac{s}{2}(Mn+A)(Mn+A-1)} (1-q^{s(Mn+A)})\Big\rvert_{q=e^{-t}} \sim M^{-1},
    \end{align*}
    where we have used the fact that $B_1(x)=x-\frac{1}{2}$. Finally, we recall the modular inversion formula for Dedekind's eta function (see, for example, \cite[p.~121, Proposition 14]{Kob1984}), which implies that as $t\to 0^+$,
    \begin{align*}
        (e^{-t};e^{-t})_\infty\sim \sqrt{2\pi}\,t^{-\frac{1}{2}}e^{-\frac{\pi^2}{6t}}.
    \end{align*}
    Hence,
    \begin{align}\label{eq:sigma-0-q-asymp}
        \sum_{n\ge 0} \sigma_{M,A}^{[s]}(0;n)e^{-nt} \sim \tfrac{1}{\sqrt{2\pi}}\,M^{-1}t^{\frac{1}{2}}e^{\frac{\pi^2}{6t}} \qquad (t\to 0^+).
    \end{align}

    When $r\ge 1$, we have
    \begin{align*}
        S_1(e^{-t}) &= A^{r}e^{-\frac{s}{2}A(A-1)t} + e^{\frac{st}{8}} \sum_{n\ge 1} (Mn+A)^r e^{-\frac{s}{2}(Mn+A-\frac{1}{2})^2 t}\\
        &= A^re^{-\frac{s}{2}A(A-1)t} + e^{\frac{st}{8}} \sum_{n\ge 0} \big[M(n+1+\tfrac{A}{M}-\tfrac{1}{2M})+\tfrac{1}{2}\big]^r e^{-\frac{sM^2}{2}(n+1+\frac{A}{M}-\frac{1}{2M})^2 t}\\
        &= 2^{\frac{r-1}{2}}M^{-1}s^{-\frac{r+1}{2}}\,\Gamma(\tfrac{r+1}{2})\, t^{-\frac{r+1}{2}} + 2^{\frac{r-4}{2}}M^{-1}s^{-\frac{r}{2}}r\,\Gamma(\tfrac{r}{2})\, t^{-\frac{r}{2}} + O(t^{-\frac{r-1}{2}}).
    \end{align*}
    Meanwhile,
    \begin{align*}
        S_2(e^{-t}) &= e^{\frac{st}{8}} \sum_{n\ge 0} \big[M(n+\tfrac{A}{M}+\tfrac{1}{2M})-\tfrac{1}{2}\big]^r e^{-\frac{sM^2}{2}(n+\frac{A}{M}+\frac{1}{2M})^2 t}\\
        &= 2^{\frac{r-1}{2}}M^{-1}s^{-\frac{r+1}{2}}\,\Gamma(\tfrac{r+1}{2})\, t^{-\frac{r+1}{2}} - 2^{\frac{r-4}{2}}M^{-1}s^{-\frac{r}{2}}r\,\Gamma(\tfrac{r}{2})\, t^{-\frac{r}{2}} + O(t^{-\frac{r-1}{2}}).
    \end{align*}
    Therefore,
    \begin{align*}
        \sum_{n\ge 0} (Mn+A)^r q^{\frac{s}{2}(Mn+A)(Mn+A-1)} (1-q^{s(Mn+A)})\Big\rvert_{q=e^{-t}} \sim 2^{\frac{r-2}{2}}M^{-1}s^{-\frac{r}{2}}r\,\Gamma(\tfrac{r}{2})\, t^{-\frac{r}{2}}.
    \end{align*}
    Now, invoking the asymptotic formula for $(e^{-t};e^{-t})_\infty$, it follows that for $r\ge 1$,
    \begin{align}\label{eq:sigma-r-q-asymp}
        \sum_{n\ge 0} \sigma_{M,A}^{[s]}(r;n)e^{-nt} \sim 2^{\frac{r-3}{2}}\pi^{-\frac{1}{2}}M^{-1}s^{-\frac{r}{2}}r\,\Gamma(\tfrac{r}{2})\, t^{\frac{1-r}{2}} e^{\frac{\pi^2}{6t}} \qquad (t\to 0^+).
    \end{align}

    Finally, with a direct application of Ingham's Tauberian theorem to \eqref{eq:sigma-0-q-asymp} and \eqref{eq:sigma-r-q-asymp}, we are led to the desired relation \eqref{eq:sigma-asymp}.
\end{proof}

Next, we move on to the proof of Theorem \ref{th:varsigma-asymp}.

\begin{proof}[Proof of Theorem \ref{th:varsigma-asymp}]
    In light of \eqref{eq:varsigma-gf-sum}, we see that
    \begin{align*}
        &\sum_{n\ge 0} \varsigma_{M,A}^{[s]}(r;n)q^n\\
        &\qquad = \frac{1}{(q;q)_\infty} \left(A^r + \sum_{n\ge 0} \big[(Mn+M+A)^r-(Mn+A)^r\big] q^{s(\frac{1}{2}Mn(n+1)+A(n+1))}\right)
    \end{align*}
    is a nonnegative series since $(Mn+M+A)^r-(Mn+A)^r$ is nonnegative for any $n\ge 0$. Meanwhile, $(1-q)\cdot \sum_{n\ge 0} \varsigma_{M,A}^{[s]}(r;n)q^n$ is also a nonnegative series because we only need to replace $(q;q)_\infty^{-1}$ with $(q^2;q)_\infty^{-1}$ in the above. Thus, the moments $\varsigma_{M,A}^{[s]}(r;n)$ are nondecreasing and nonnegative.

    Now we are left to evaluate $\sum_{n\ge 0} \varsigma_{M,A}^{[s]}(r;n)e^{-nt}$ as $t\to 0^+$. It is clear from \eqref{eq:varsigma-gf-sum} that
    \begin{align*}
        &\sum_{n\ge 0} (Mn+A)^r q^{s(\frac{1}{2}Mn(n-1)+An)} (1-q^{s(Mn+A)})\\
        &\qquad = A^r + q^{-\frac{sM}{2}(\frac{A}{M}-\frac{1}{2})^2}\sum_{n\ge 0} \big[M(n+\tfrac{A}{M}+\tfrac{1}{2})+\tfrac{M}{2}\big]^r q^{\frac{sM}{2}(n+\frac{A}{M}+\frac{1}{2})^2}\\
        &\qquad\quad - q^{-\frac{sM}{2}(\frac{A}{M}-\frac{1}{2})^2}\sum_{n\ge 0} \big[M(n+\tfrac{A}{M}+\tfrac{1}{2})-\tfrac{M}{2}\big]^r q^{\frac{sM}{2}(n+\frac{A}{M}+\frac{1}{2})^2}.
    \end{align*}
    Hence, an application of Lemma \ref{le:asymp-partial} gives
    \begin{align*}
        \sum_{n\ge 0} (Mn+A)^r q^{s(\frac{1}{2}Mn(n-1)+An)} (1-q^{s(Mn+A)})\Big\rvert_{q=e^{-t}} \sim 2^{\frac{r-2}{2}}M^{\frac{r}{2}}s^{-\frac{r}{2}}r\,\Gamma(\tfrac{r}{2})\, t^{-\frac{r}{2}}.
    \end{align*}
    We further invoke the contribution from $(e^{-t};e^{-t})_\infty^{-1}$ and get
    \begin{align}\label{eq:varsigma-r-q-asymp}
        \sum_{n\ge 0} \varsigma_r(a,b;n)e^{-nt} \sim 2^{\frac{r-3}{2}}\pi^{-\frac{1}{2}}M^{\frac{r}{2}}s^{-\frac{r}{2}}r\,\Gamma(\tfrac{r}{2})\, t^{\frac{1-r}{2}} e^{\frac{\pi^2}{6t}} \qquad (t\to 0^+).
    \end{align}

    Finally, \eqref{eq:varsigma-asymp} follows by applying Ingham's Tauberian theorem to \eqref{eq:varsigma-r-q-asymp}.
\end{proof}

\section{Conclusion}

It should be noted that Ingham's Tauberian theorem only gives the
dominant term of the asymptotic estimation in question, while it
talks nothing about the errors. Such a type of estimation is
sufficient to demonstrate asymptotic ``equalities'' like our
Corollaries \ref{eq:sigma-asymp-eq} and \ref{eq:varsigma-asymp-eq}.
However, when it comes to asymptotic ``inequalities'' such as the
eventual log-concavity or bias of $\sigma_{M,A}^{[s]}(r;n)$ and
$\varsigma_{M,A}^{[s]}(r;n)$ as inquired below, a more precise
asymptotic expansion becomes necessary. For this purpose, we expect
a delicate application of the circle method, which will be left for
future research.

\begin{problem}[Log-concavity]
    Fix $r$, $s$, $A$ and $M$. Is it true for all sufficiently large $n$ that
    \begin{align}
        \sigma_{M,A}^{[s]}(r;n)^2 &> \sigma_{M,A}^{[s]}(r;n-1)\sigma_{M,A}^{[s]}(r;n+1),\\
        \varsigma_{M,A}^{[s]}(r;n)^2 &> \varsigma_{M,A}^{[s]}(r;n-1)\varsigma_{M,A}^{[s]}(r;n+1).
    \end{align}
\end{problem}

\begin{problem}[Bias]
    Fix $r$, $s$ and $M$. Is there a reordering $A_1,A_2,\ldots,A_M$ of $1,2,\ldots, M$ such that for all sufficiently large $n$,
    \begin{align}
        \sigma_{M,A_1}^{[s]}(r;n)\le \sigma_{M,A_2}^{[s]}(r;n)\le \cdots \le \sigma_{M,A_M}^{[s]}(r;n).
    \end{align}
    The same question may also be asked for $\varsigma_{M,A}^{[s]}(r;n)$.
\end{problem}

\subsection*{Acknowledgements}

Ernest X. W. Xia was supported by the National Natural Science
Foundation of China (no.~12371334) and the Natural Science
Foundation of Jiangsu Province of China (no.~BK20221383).

   \noindent{\bf Data Availability
   Statements.} Data
sharing not applicable to this article as no datasets were generated
or
 analysed during the current study.

 \noindent{\bf Declaration of
  Competing Interest.}
 The authors  declared that  they have
  no
 conflicts of interest to this work.

\bibliographystyle{amsplain}

\end{document}